\documentclass{amsart}
\usepackage{main}

 \title[Breakdown of regularity of scattering] 
 {Breakdown of regularity of scattering for mass-subcritical NLS}

\author[Gyu Eun Lee]{}


\begin{document}

\maketitle

\centerline{\scshape Gyu Eun Lee}
\medskip
{\footnotesize
 \centerline{Department of Mathematics, University of California, Los Angeles}
   \centerline{ Los Angeles, CA 90095, USA}
} 

\bigskip

\begin{abstract}
    We study the scattering problem for the nonlinear Schr\"odinger equation $i\pt_t u + \Delta u = |u|^p u$ on $\mathbb{R}^d$, $d\geq 1$, with a mass-subcritical nonlinearity above the Strauss exponent.
    For this equation, it is known that asymptotic completeness in $L^2$ with initial data in $\Sigma$ holds and the wave operator is well-defined on $\Sigma$.
    We show that there exists $0<\beta<p$ such that the wave operator and the data-to-scattering-state map do not admit extensions to maps $L^2\to L^2$ of class $C^{1+\beta}$ near the origin.
    This constitutes a mild form of ill-posedness for the scattering problem in the $L^2$ topology.
\end{abstract}


\section{Introduction}\label{sec:intro}

Consider the \ita{defocusing mass-subcritical nonlinear Schr\"odinger equation} (NLS):
\begin{equation}\label{eqn:NLS_Cauchy}
        i\pt_t u + \Delta u = F(u) = |u|^p u, ~(t,x)\in I\times\bb{R}^d \subset \bb{R}\times\bb{R}^d,
\end{equation}
where $0 < p < \frac{4}{d}$.
It is well-known that the Cauchy problem for this equation is globally well-posed in $L^2$.
In this paper we are concerned with two elements of the long-time asymptotic behavior of solutions to this equation in the $L^2$ topology.
The first is the question of asymptotic completeness.
We say that \eqref{eqn:NLS_Cauchy} is \ita{asymptotically complete} in $L^2$ if for each $\phi\in L^2$, there exists $u_+\in L^2$ so that the global solution $u$ to \eqref{eqn:NLS_Cauchy} with $u(t=0) = \phi$ satisfies
\[
    \lim_{t\to\infty} \|e^{-it\Delta}u(t) - u_+ \|_{L^2} = 0.
\]
The second is the existence of the wave operator.
We say that the \ita{wave operator} for \eqref{eqn:NLS_Cauchy} is well-defined on $L^2$ if for each $\phi\in L^2$, there exists a unique global solution $u\in C_{t,\tnm{loc}}L_x^2$ to \eqref{eqn:NLS_Cauchy} satisfying
\[
    \lim_{t\to\infty} \|e^{-it\Delta}u(t) - \phi\|_{L^2} = 0.    
\]
Analogous definitions can be made as $t\to-\infty$; as the distinction between forward and backward time does not affect any part of this paper, we consider the forward time direction only.

Whether \eqref{eqn:NLS_Cauchy} is asymptotically complete or admits a wave operator in $L^2$ are currently open problems.
The known results rely on stronger assumptions on the space of initial data or scattering states.
We introduce the two representative results here.
Let $\Sigma$ be the Banach space defined by the norm
\[
    \|f\|_\Sigma^2 = \|f\|_{L^2}^2 + \|\nabla f\|_{L^2}^2 + \|xf\|_{L^2}^2.    
\]
The first result is by Ginibre and Velo, and establishes the scattering theory in $\Sigma$.
\begin{thm}[Scattering in $\Sigma$; \cites{GiVe79.Cauchy, GiVe19.Scattering}]\label{thm:Sigma_scattering}
    \hspace{1em}
    \begin{enumerate}
        \item Let $\alpha(d) < p < \frac{4}{d}$, where
        \[
            \alpha(d) = \frac{2-d + \sqrt{(d-2)^2+16d}}{2d}
        \]
        denotes the \ita{Strauss exponent}
        \footnote{The Strauss exponent is a current technical limitation for the scattering theory for \eqref{eqn:NLS_Cauchy}.
        It represents the threshold at which one can obtain global spacetime bounds for the solution in critically scaling Strichartz spaces.
        We note that the range of nonlinearities for Theorem \ref{thm:Sigma_scattering} be broadened to $\frac{4}{d+2}<p<\frac{4}{d}$ with a small-data assumption, due to Cazenave and Weissler \cites{CaWe92}.}
        Then the Cauchy problem for \eqref{eqn:NLS_Cauchy} is globally well-posed in $\Sigma$.
        \item \eqref{eqn:NLS_Cauchy} is asymptotically complete in $\Sigma$: for each $\phi\in\Sigma$, there exists $u_+\in \Sigma$ so that the global solution $u$ to \eqref{eqn:NLS_Cauchy} with $u(t=0) = \phi$ satisfies
        \[
            \lim_{t\to\infty} \|e^{-it\Delta}u(t) - u_+ \|_{\Sigma} = 0.
        \]
        \item The wave operator is well-defined on $\Sigma$; for each $\phi\in \Sigma$, there exists a unique global solution $u\in C_{t,\tnm{loc}}\Sigma(\mathbb{R})$ to \eqref{eqn:NLS_Cauchy} satisfying
        \[
            \lim_{t\to\infty} \|e^{-it\Delta}u(t) - \phi\|_{\Sigma} = 0.    
        \]
    \end{enumerate}
\end{thm}
The second result is due to Tsutsumi and Yajima:
\begin{thm}[Asymptotic completeness in $L^2$ for $\Sigma$ data; \cites{TsYa84}]\label{thm:Sigma_L2_completeness}
    Let $\frac{2}{d} < p < \frac{4}{d}$.
    Then for each $\phi\in\Sigma$, there exists $u_+\in L^2$ so that the global solution $u$ to \eqref{eqn:NLS_Cauchy} with $u(t=0) = \phi$ satisfies
    \[
        \lim_{t\to\infty} \|e^{-it\Delta}u(t) - u_+ \|_{L^2} = 0.
    \]
\end{thm}
These results do not address the question of taking data in $L^2$, which is arguably the most natural space for the problem given its mass-subcritical nature and the conservation of mass.
In this paper we offer something of an explanation for this state of affairs.
We now define our main objects and state our main results.
\begin{dfn}
    Let $\alpha(d) < p < \frac{4}{d}$.
    The \ita{initial-to-final-state map} is the map $\mcal{S}:\Sigma\to L^2$ defined by $\mcal{S}(\phi) = \lim_{t\to\infty} e^{-it\Delta}u(t) = u_+$, where $u\in C_{t,\tnm{loc}}L_x^2$ is the global solution to \eqref{eqn:NLS_Cauchy} and the limit is in the $L^2$ topology.
    The \ita{wave operator} is the map $\mcal{W}:\Sigma\to L^2$ defined by $\mcal{W}(\phi) = u(0)$, where $u\in C_{t,\tnm{loc}}L_x^2$ is the unique global solution to \eqref{eqn:NLS_Cauchy} satisfying $\lim_{t\to\infty} \|e^{-it\Delta}u(t) - \phi\|_{L^2} = 0$.
\end{dfn}
Note that $\mcal{S}$ and $\mcal{W}$ are well-defined by Theorem \ref{thm:Sigma_scattering}; in fact, we could take them to be $\Sigma$-valued, but this is not necessary for our purposes.
\begin{thm}[Main theorem]\label{thm:main}
    Assume $d\geq 1$ and $\alpha(d) < p < \frac{4}{d}$.
    Then:
    \begin{enumerate}
        \item $\mcal{S}$ and $\mcal{W}$, regarded as maps $\Sigma\to L^2$, are $s$-H\"older continuous at $0\in\Sigma$ for all $0<s\leq 1+p$ at $0$, and are not $s$-H\"older continuous at $0$ for any $s>1+p$.
        \item There exists $0<\beta<p$, with $\beta$ depending only on $d$ and $p$, such that for any ball $B\subset L^2$ containing the origin, $\mcal{S}:B\cap\Sigma\to L^2$ and $\mcal{W}:B\cap\Sigma\to L^2$ cannot be extended to maps $B\subset L^2\to L^2$ which are H\"older continuous of order $1+\beta$ at $0\in L^2$.
    \end{enumerate}
\end{thm}
Here, for $s>0$ possibly non-integer, by a map $F:X\to L^2$ (where $X = \Sigma$ or $L^2$) that is $s$-H\"older continuous of order $s$ at $x_0 \in X$ we mean a map which belongs to the pointwise H\"older space $C^s(x_0)$ (see Definition \ref{dfn:Holder_space}).
In particular:
\begin{cor}\label{cor:main}
    Assume $d\geq 1$ and $\alpha(d) < p < \frac{4}{d}$.
    Then:
    \begin{enumerate}
        \item Let $s>1+p$, and let $n$ be the integer part of $s$.
        Then $\mcal{S}$ and $\mcal{W}$, regarded as maps $\Sigma\to L^2$, cannot have an $n$-th Gateaux derivative defined about $0\in\Sigma$ which is H\"older continuous of order $s-n$.
        \item Let $s = 1+\beta$, where $\beta$ is as in Theorem \ref{thm:main}, and let $n$ be the integer part of $s$.
        Then $\mcal{S}$ and $\mcal{W}$ cannot be extended to maps $L^2\to L^2$ that admit an $n$-th Gateaux derivative defined about $0\in L^2$ which is H\"older continuous of order $s-n$.
    \end{enumerate}
\end{cor}

Part (1) of Theorem \ref{thm:main} is, in some sense, unsurprising: since the nonlinearity in \eqref{eqn:NLS_Cauchy} is a pure power of degree $1+p$, given a sufficiently strong global well-posedness and scattering theory we should expect to be able to differentiate with respect to the initial or final state up to, and not more than, $1+p$ times.
In that sense part (1) provides the sharp regularity result with respect to the nonlinearity.
That this amount of regularity holds with initial or final states in $\Sigma$ is a confirmation that $\Sigma$ is such a space with a strong global well-posedness and scattering theory.
Therefore part (2) of Theorem \ref{thm:main} is the statement of primary interest: it states that if we instead take $L^2$ as our space of initial or final states, then the initial-to-final state operator and the wave operator, if they were to be defined on $L^2$, cannot attain the regularity with respect to the data suggested by the smoothness of the nonlinearity.
The value of $\beta$ in part (2) can be made explicit, which will become evident toward the end of the proof of Theorem \ref{thm:main}.
We interpret this as a mild form of ill-posedness result for the asymptotic completeness and wave operator problems in the $L^2$ topology.

We briefly review the history and relevant work behind this result.
We have already mentioned the two main positive results in the scattering theory for the mass-subcritical NLS: the work of Ginibre-Velo \cites{GiVe79.Cauchy,GiVe19.Scattering}, which establishes the scattering theory in $\Sigma$, and that of Tsutsumi-Yajima \cites{TsYa84}, which establishes asymptotic completeness in $L^2$ under the assumption of $\Sigma$ data.
The result of Tsutsumi-Yajima is optimal in that it treats the full range of nonlinearities $\frac{2}{d} < p < \frac{4}{d}$ (the so-called \ita{short-range} regime) for which mass-subcritical scattering in $L^2$ is possible; for $p\leq\frac{2}{d}$ (the \ita{long-range} regime), scattering in $L^2$ can only occur for the zero solution, which is a result due to Strauss, Glassey, and Barab \cites{Strauss73, Glassey73, Barab84}.
To be clear, we are not asserting that $\Sigma$ is a purely artificial space for the scattering theory.
It is in fact a very natural space: after a Lens transformation it arises as the harmonic energy space for Equation \ref{eqn:NLS_Cauchy}; see \cites{Tao09}.

The corresponding literature for mass-subcritical scattering with data in $L^2$ is sparser.
There is one positive result due to Nakanishi \cites{Na01}: in the full short-range regime, for any free evolution $e^{it\Delta}\psi$ there is a global solution $u$ to \eqref{eqn:NLS_Cauchy} which approximates it in $L^2$ (resp. in $H^1$) as $t\to\infty$.
However, it is not known if the global solution thus obtained is unique, so this falls short of defining the wave operator on $L^2$.
Moreover, this global solution is obtained by compactness methods, and we do not obtain a quantitative understanding of how it depends on the final state $\psi$.
In fact, the asymptotic completeness result of Tsutsumi-Yajima also proceeds by a compactness argument, and thus we do not obtain a quantitative understanding of the initial-to-final-state map either.
As for the mass-critical case $p = \frac{4}{d}$, global well-posedness and scattering in $L^2$ are known due to recent work of Dodson \cites{Do12, Do16.d1GWP, Do16.d2GWP}.
The scattering theory for the mass-critical NLS is essentially a direct consequence of its global well-posedness theory, as it admits a symmetry under the pseudoconformal transformation.

Perhaps it is telling that despite the question being relatively obvious to any student of the subject, there have been few if any positive results in the direction of mass-subcritical scattering in $L^2$.
In fact, the fact that we can construct initial states for given final states is already somewhat remarkable.
This is because (as observed by Nakanishi) the wave operator problem is in some sense scaling-supercritical in $L^2$, as can be seen by applying the pseudoconformal transform to convert it into an initial value problem.
Generally speaking, the problem of constructing the scattering operator is at least as difficult as that of constructing the wave operator.
This is because when the Cauchy problem is posed in a subcritical or critical space, it is often the case that the wave operator can be constructed by Picard iteration in the same way as one constructs local-in-time solutions.
To conclude asymptotic completeness, one additionally needs some sort of decay estimate on the nonlinear evolution that is consistent with the dispersive decay of the linear evolution.
Therefore the question of asymptotic completeness in $L^2$ appears to require both a well-posedness theory for a scaling-supercritical problem and rather explicit decay-in-time estimates.
It is generally conjectured that scaling-supercritical problems exhibit some form of ill-posedness.
This is the primary motivation behind this paper: to demonstrate that the asymptotic completeness and wave operator problems on $L^2$ are ill-posed in some appropriate mild sense, offering a partial explanation for the difficulty in resolving these problems.

We now outline our methods for proving Theorem \ref{thm:main}.
Our ill-posedness argument proceeds along the following abstract framework:
\begin{enumerate}
    \item Decompose the solution under consideration into an explicit main term and an error term, possibly in a stronger topology than where ill-posedness is to be proved (in order to make this splitting possible in the first place).
    \item Demonstrate that when restricted to the weaker topology, the main term exhibits the desired ill-posedness properties.
    \item Show that in the regime that the main term exhibits ill-posedness, the error term is dominated by the main term; therefore the error does not destroy the ill-posedness.
\end{enumerate}
This abstract framework has been used previously to prove ill-posedness properties (e.g. norm inflation) of various initial-value problems for NLS: see, for instance, \cites{ChCoTa03, BeTa06, Ki18}.
The result itself falls into a class of results sometimes known as $C^k$ or analytic ill-posedness, in which it is shown that a data-to-solution map of some type lacks regularity with respect to the data.
As far as we are aware, this class of result was first investigated큰냄비 없으면 내꺼 가져올게 by Bourgain in \cites{Bo97}, in which the threshold Sobolev regularity for smoothness of the periodic KdV flow was determined.

The key tool in our analysis, corresponding to step (1), is the following small-data expansion of the scattering state near the origin:
\begin{equation}\label{eqn:expansion}
    \mcal{S}(\phi) = \phi - i \int_0^\infty e^{-is\Delta} F(e^{is\Delta}\phi)~ds + \tnm{error}
\end{equation}
for small Cauchy data $\phi$..
A similar expansion holds for the wave operator $\mcal{W}$, and the following discussion for $\mcal{S}$ holds equally for $\mcal{W}$, so for now let us speak only of $\mcal{S}$ for brevity.
We obtain \eqref{eqn:expansion} as an explication of the proof of Theorem \ref{thm:Sigma_scattering}; this is the reason for the restriction $p > \alpha(d)$.
The first term in this expansion arises from the linear evolution; the second term arises naturally as the first nonlinear term in the Picard iteration scheme for $u$.
We will formulate this expansion precisely, prove that it holds, and quantify the error term.
Part (1) of Theorem \ref{thm:main} will then emerge almost immediately as a consequence.
This sort of expansion appeared for the mass-critical NLS in \cites{CaOz08}, with a different proof.

Next we outline step (2). Assuming that the error is negligible compared to the remaining terms, this allows us to write
\[
    \mcal{S}(\phi) - \phi \sim - i \int_0^\infty e^{-is\Delta} F(e^{is\Delta}\phi)~ds.
\]
With a sufficiently strong estimate on the error term, it is easily seen that $\mcal{S}$ is differentiable at $0$ in the Fr\'echet sense, with derivative the identity operator; therefore the left-hand side of the above expression is exactly equal to the remainder term in a first-order Taylor approximation of the scattering operator near the origin.
Therefore the behavior of quotients such as
\[
    \frac{\|\mcal{S}(\phi) - \phi\|_{L^2}}{\|\phi\|_{L^2}^{1+\beta}}  
\]
as $\|\phi\|_{L^2}\to 0$ corresponds to the regularity (or lack thereof) of the scattering map beyond the first derivative.

Due to the small-data expansion of the scattering state, by $L^2$ duality the numerator in the above quotient is essentially equivalent to
\begin{align*}
    \left\| \int_0^\infty e^{-is\Delta} F(e^{is\Delta}\phi)~ds \right\|_{L^2} 
        &\geq \frac{1}{\|\phi\|_{L^2}} \left| \int_0^\infty \langle e^{-is\Delta} F(e^{is\Delta}\phi), \phi \rangle_{L_x^2}~ds\right|\\
        &= \frac{1}{\|\phi\|_{L^2}}\|e^{it\Delta}\phi\|_{L_{t,x}^{p+2}([0,\infty))}^{p+2}.
\end{align*}
We would therefore like to show that
\[
    \sup_{\phi\in B_R}\frac{\|e^{it\Delta}\phi\|_{L_{t,x}^{p+2}([0,\infty))}^{p+2}}{\|\phi\|_{L^2}^{2+\beta}} = \infty.
\]
To do this, it is required that $\|e^{it\Delta}\phi\|_{L_{t,x}^{p+2}([0,\infty))}$ cannot be controlled by large powers of the mass $\|\phi\|_{L^2}$.
This is a familiar result: such control can only be obtained in the mass-critical case $p = \frac{4}{d}$, for which it is the $L^2$ Strichartz estimate at the Tomas-Stein exponent, and in all other cases one can show by scaling that there exists an $L^2$-bounded sequence $(\phi_n)$ for which $\|e^{it\Delta}\phi_n\|_{L_{t,x}^{p+2}([0,\infty))} \to \infty$.
That the failure of a Strichartz estimate can lead to conclusions about the regularity of a data-to-solution map essentially dates back to \cites{Bo97}; see also \cites{Tao06.Dispersive} for a textbook treatment.

Therefore the main term exhibits the desired ill-posedness properties, and only step (3) remains: we must find a sequence of initial data $\phi$ which sends the above quotient to $\infty$, while ensuring that the main term dominates the error term.
We do this by introducing a two-parameter family of initial data $\phi_{\ve,\sigma}$, where $\ve$ is the amplitude and $\sigma$ is the spatial scale.
The error term is essentially a higher-order term in $\ve$ compared to the main term; therefore we can choose $\ve$ small so that the main term defeats the error, and then $\sigma$ to send the quotient to $\infty$.

We now briefly outline the organization of the paper.
In Section \ref{sec:notation} we go over the notation and basic preliminary results used in the rest of the paper.
In Section \ref{sec:expansion} we prove that $\mcal{S}$ and $\mcal{W}$ admit the expansion \eqref{eqn:expansion} with quantitative bounds on the error term.
In the interest of exposition, we do this in dimensions $d\geq 4$ only.
In this setting, the nonlinearity is subquadratic, which simplifies some technical details.
In Section \ref{sec:illposedness}, we leverage \eqref{eqn:expansion} with the error estimates to demonstrate that the ill-posedness properties of the main term carry over to ill-posedness of the entire operator.
In Appendix \ref{app:dim_1_proof} we show how to recover the cases $d=1,2,3$.

\section{Notation and preliminaries}\label{sec:notation}

Let $X$ and $Y$ be two quantities.
We write $X\lesssim Y$ if there exists a constant $C>0$ such that $X\leq CY$.
If $C$ depends on parameters $a_1,\ldots,a_n$, i.e. $C = C(a_1,\ldots,a_n)$ and we wish to indicate this dependence, then we will write $X\lesssim_{a_1,\ldots,a_n} Y$.
If $X\lesssim Y$ and $Y\lesssim X$, we write $X\sim Y$.
If the constant $C$ is small, then we write $X \ll Y$.
We also employ the asymptotic notation $\mcal{O}(f)$ and $o(f)$ with their standard meanings.

We will be working with the mixed spacetime Lebesgue spaces $L_t^qL_x^r(I\times\bb{R}^d)$, with norms
\[
    \|u\|_{L_t^qL_x^r(I\times\bb{R}^d)} = \left( \int_I \left( \int_{\bb{R}^d} |u(t,x)|^r~dx\right)^{\frac{q}{r}}~dt\right)^{\frac{1}{q}}.    
\]
We will abbreviate the norm as $\|u\|_{L_t^qL_x^r(I\times\bb{R}^d)} = \|u\|_{L_t^qL_x^r(I)}$.
We will often encounter the case $I = [0,\infty)$.
In this case we will further abbreviate the norm as $\|u\|_{L_t^qL_x^r([0,\infty)\times\bb{R}^d)} = \|u\|_{q,r}$.
For functions with no time dependence, we write $\|f\|_{L^r(\bb{R}^d)} = \|f\|_r$.
For $1\leq r \leq \infty$, we denote by $r'$ the H\"older conjugate: $1 = \frac{1}{r} + \frac{1}{r'}$.

We define the energy functional
\[
    E(v) = \int_{\bb{R}^d} \frac{1}{2}|\nabla v|^2 + \frac{1}{p+2}|v|^{p+2}~dx.    
\]
It is well known that the energy, as well as the $L^2$-norm, are conserved quantities for solutions to \eqref{eqn:NLS_Cauchy}.

We recall the following fundamental estimates for the Schr\"odinger equation.
\begin{prop}[Dispersive estimate]\label{prop:dispersive_estimate}
    Let $2\leq r \leq\infty$.
    Then for all $t\neq 0$,
    \[
        \|e^{it\Delta}\phi\|_{L_x^r(\bb{R}^d)} \lesssim_{r,d} |t|^{\frac{d}{2} - \frac{d}{r}}\|\phi\|_{L^{r'}(\bb{R}^d)}.
    \]
\end{prop}
\begin{dfn}[Admissible pair]\label{dfn:admissible}
    Let $d\geq 1$ and $2 \leq q,r \leq \infty$.
    We say that $(q,r)$ is an \ita{admissible pair} if it satisfies the scaling relation $\frac{2}{q} + \frac{d}{r} = \frac{d}{2}$ and $(d,q,r)\neq(2,2,\infty)$.
    We say that $(\alpha,\beta)$ is a \ita{dual admissible pair} if $(\alpha',\beta')$ is an admissible pair.
\end{dfn}
\begin{prop}[Strichartz estimates]\label{prop:Strichartz}
    Let $d\geq 1$, let $(q,r)$ be an admissible pair, and let $(\alpha,\beta)$ be a dual admissible pair.
    Then for any interval $I\subset\bb{R}$,
    \[
        \|e^{it\Delta}\phi\|_{L_t^qL_x^r(I\times\bb{R}^d)} \lesssim \|\phi\|_{L^2(\bb{R}^d)},
    \]
    \[
        \left\| \int_0^t e^{i(t-s)\Delta} F(s)~ds \right\|_{L_t^qL_x^r(I\times\bb{R}^d)} \lesssim \|F\|_{L_t^\alpha L_x^\beta(I\times\bb{R}^d)}.    
    \]
\end{prop}

Lastly, we define our notion of pointwise H\"older regularity.
\begin{dfn}[Pointwise H\"older space \cites{Andersson97}]\label{dfn:Holder_space}
    Let $X$ and $Y$ be Banach spaces.
    Let $x_0\in X$ and $U$ a convex open neighborhood of $x_0$.
    Fix $s>0$, and let $n$ be the integer part of $s$.
    For $s>0$, we say that the map $G:X\to Y$ belongs to the \ita{pointwise H\"older space} $C^s(x_0)$ if for all $h\in X$ with $\|h\|_X = 1$, there exist coefficients $\{a_j(x_0; h)\}_{j=0}^n \subset Y$ such that
    \[
        \|G(x_0 + \ve h) - G(x_0) - \sum_{j=1}^n \ve^j a_j(x_0; h)\|_Y \lesssim \ve^s   
    \]
    for all $\ve>0$ sufficiently small, with the implicit constant independent of the direction $h$.
\end{dfn}
Our main interest in $C^s(x_0)$ is that membership in $C^s(x_0)$ is a necessary, though not sufficient, condition for a stronger notion of regularity of order $s$:
\begin{lem}\label{lem:Cs_necessary}
    Let $X$ and $Y$ be Banach spaces.
    Let $U\subset X$ be a convex neighborhood of $x_0\in X$.
    Let $G:U\to Y$ be a map, and suppose $G\notin C^s(x_0)$ with $n < s < n+1$.
    Then $d^nG(x;h)$ (the $n$-th Gateaux derivative of $G$), if it exists for $x\in U$, cannot be a H\"older continuous function of $x$ of order $s-n$ with H\"older seminorm uniformly bounded in $h$.
\end{lem}
For the proof of Lemma \ref{lem:Cs_necessary}, as well as the relationship between Definition \ref{dfn:Holder_space} and other notions of regularity, we refer the reader to Appendix \ref{app:Holder_spaces}.

\section{Small-data expansion of the wave and initial-to-scattering-state operators}\label{sec:expansion}

In this section we undertake step (1) of our ill-posedness framework, the decomposition of the wave operator and the initial-to-scattering-state map.

Henceforth we take $q = \frac{4(p+2)}{dp}$; then $(q,p+2)$ is an admissible pair.
We write $\mcal{T}_- = \mcal{S}$, $\mcal{T}_+ = \mcal{W}$, regarding them as maps $\mcal{T}_\pm:\Sigma\to L^2$.
\begin{prop}[Small-data expansion]\label{prop:small_data_expansion}
    Let $\alpha(d) < p < \frac{4}{d}$.
    Then there exists $\ve = \ve(d,p)>0$ small so that if $\|\phi\|_\Sigma < \ve$, then
    \[
        \mcal{T}_\pm(\phi) = \phi \pm i \int_0^\infty e^{-is\Delta}F(e^{is\Delta}\phi)~ds + e_\pm(\phi)
    \]
    where the error term $e_\pm(\phi)$ satisfies
    \begin{equation}\label{eqn:error_estimate}
        \|e_\pm(\phi)\|_2 \lesssim_{d,p} \|\phi\|_\Sigma^{\frac{2(2p+1)}{p+2}}.
    \end{equation}
\end{prop}
The proof of Proposition \ref{prop:small_data_expansion} is based on the proof of Theorem \ref{thm:Sigma_scattering} given in \cites{CaWe92}.
The key fact underlying all of these arguments is the pseudoconformal energy estimate:
\begin{lem}[Pseudoconformal energy estimate]\label{lem:effective_pseudoconformal}
    Let $0 < p < \frac{4}{d}$ and $\phi\in\Sigma$.
    Let $u$ be the global solution to \eqref{eqn:NLS_Cauchy} with initial data $\phi$.
    Then for all $t\geq 1$,
    \begin{equation}\label{eqn:effective_pcf_general}
        \|u(t)\|_{L_x^{p+2}}^{p+2} \lesssim_{d,p} t^{-\frac{dp}{2}}(\|u(1)\|_\Sigma^2 + \|u(1)\|_\Sigma^{p+2}).    
    \end{equation}
    Moreover, if $\ve = \ve(d,p)>0$ is sufficiently small and $\|\phi\|_\Sigma < \ve$, then for all $t\geq 0$
    \begin{equation}\label{eqn:effective_pcf_small}
        \|u(t)\|_{L_x^{p+2}}^{p+2} \lesssim_{d,p} \langle t\rangle^{-\frac{dp}{2}}\|\phi\|_\Sigma^2.  
    \end{equation}
\end{lem}
The proof is well-known, though it is usually stated without the explicit dependence on $u(1)$ or the small-data statement.
We reproduce it here for the reader's convenience; our proof follows the presentation in \cites{Murphy.scattering}.
\begin{proof}
    Let $J(t) = x+it\nabla$.
    \eqref{eqn:effective_pcf_general} follows from the more general inequality
    \[
        \|J(t)u(t)\|_{L_x^2}^2 + \|u(t)\|_{L_x^{p+2}}^{p+2} \lesssim_{d,p} t^{-\frac{dp}{2}}(\|u(1)\|_\Sigma^2 + \|u(1)\|_\Sigma^{p+2}).
    \]
    Expanding,
    \[
        \|J(t)u(t)\|_{L_x^2}^2 = \int |x|^2|u|^2 - 2t\Im(\ol{u}\nabla u\cdot 2x) + 4t^2|\nabla u|^2~dx.    
    \]
    Next, we invoke the virial identity for solutions to \eqref{eqn:NLS_Cauchy}:
    \[
        \frac{d^2}{dt^2}\int |x|^2|u|^2~dx = \frac{d}{dt}2\Im \int \ol{u}\nabla u\cdot 2x~dx = \int \frac{4dp}{p+2}|u|^{p+2} + 8|\nabla u|^2~dx.
    \]
    It follows that
    \[
        \frac{d}{dt}\int |x|^2|u|^2 - 2t\Im \ol{u}\nabla u\cdot 2x~dx = -\int \frac{4dpt}{p+2}|u|^{p+2} + 8t|\nabla u|^2~dx. 
    \]
    By conservation of energy,
    \[
        4t^2\frac{d}{dt}\int |\nabla u|^2 ~dx = -8t^2\frac{d}{dt}\int \frac{1}{p+2}|u|^{p+2}~dx.    
    \]
    Combining, we obtain
    \[
        \frac{d}{dt}\int |J(t)u(t)|^2~dx = - \int \frac{4dpt}{p+2}|u|^{p+2} - 8t^2\frac{d}{dt}\int \frac{1}{p+2}|u|^{p+2}~dx.    
    \]
    Defining
    \[
        e(t) = \int |J(t)u(t)|^2 + \frac{8t^2}{p+2}|u|^{p+2}~dx,    
    \]
    we find that
    \[
        \dot{e}(t) = \frac{4t(4-dp)}{p+2}\int|u|^{p+2}~dx = \frac{2-\frac{dp}{2}}{t}\frac{8t^2}{p+2}\int |u|^{p+2}~dx.    
    \]
    With
    \[
        U(t) = \frac{8t^2}{p+2}\int |u|^{p+2}~dx,    
    \]
    it follows that
    \[
        U(t) \leq e(t) = e(1) + \int_1^t \dot{e}(s)~ds = e(1) + \int_1^t \frac{2-\frac{dp}{2}}{s}U(s)~ds.    
    \]
    By Gronwall's inequality and Sobolev embedding, we conclude that
    \[
        U(t) 
            \leq e(1)\exp\left( \int_1^t \frac{2-\frac{dp}{2}}{s}~ds\right) 
            \lesssim (\|u(1)\|_\Sigma^2 + \|u(1)\|_\Sigma^{p+2})t^{2-\frac{dp}{2}}
    \]
    which gives the claim.
    The small-data statement now follows from the local well-posedness theory for NLS in $\Sigma$.
\end{proof}
Proposition \ref{prop:small_data_expansion} then follows from the following two estimates:
\begin{lem}\label{lem:global_nonlinear_estimate}
    Let $d\geq 1$, $\alpha(d) < p < \frac{4}{d}$, and $\phi\in\Sigma$.
    Then there exists $\ve = \ve(d,p)>0$ small so that if $\|\phi\|_\Sigma < \ve$, then
    \[
        \|u\|_{\frac{pq}{q-2},p+2} + \|u\|_{q,p+2} \lesssim_{d,p} \|\phi\|_\Sigma^{\frac{2}{p+2}}.
    \]
\end{lem}
\begin{proof}
    By Lemma \ref{lem:effective_pseudoconformal},
    \begin{align*}
        \|u\|_{\frac{pq}{q-2},p+2} 
            &\lesssim \left( \int_0^\infty (\langle t\rangle^{-\frac{dp}{2(p+2)}} \|\phi\|_\Sigma^{\frac{2}{p+2}})^{\frac{pq}{q-2}} ~dt \right)^{\frac{q-2}{pq}}\\
            &= \|\phi\|_\Sigma^{\frac{2}{p+2}} \left( \int_0^\infty \langle t\rangle^{-\frac{2p}{q-2}}~dt \right)^{\frac{q-2}{pq}}.
    \end{align*}
    The integral in time is finite provided $\frac{2p}{q-2} > 1$, which is true whenever $p > \alpha(d)$.
    A similar argument shows that
    \[
        \|u\|_{q,p+2} \lesssim \|\phi\|_\Sigma^{\frac{2}{p+2}} \left( \int_0^\infty \langle t\rangle^{-2}~dt\right)^{\frac{1}{q}} \lesssim \|\phi\|_\Sigma^{\frac{2}{p+2}}.
    \]
\end{proof}
\begin{lem}\label{lem:global_linear_estimate}
    Let $d\geq 1$, $\alpha(d) < p < \frac{4}{d}$, and $\phi\in\Sigma$.
    Then
    \[
        \|e^{it\Delta}\phi\|_{\frac{pq}{q-2},p+2} \lesssim_{d,p} \|\phi\|_\Sigma.
    \]
\end{lem}
\begin{proof}
    The dispersive estimate (Proposition \ref{prop:dispersive_estimate}), combined with the Gagliardo-Nirenberg inequality and the embedding $\Sigma\hookrightarrow L^q$ ($\frac{2d}{d+2} < q \leq 2$), gives us the decay estimate
    \[
        \|e^{it\Delta}\phi\|_{p+2} \lesssim \langle t\rangle^{-\frac{dp}{2(p+2)}}\|\phi\|_\Sigma.
    \]
    From here the proof is nearly identical to that of Lemma \ref{lem:global_nonlinear_estimate}, where we invoke the above decay estimate in place of the pseudoconformal energy estimate.
\end{proof}
\begin{proof}[Proof of Proposition \ref{prop:small_data_expansion}]
    By the construction of the wave operator given in \cites{CaWe92}, for a given final state $\phi\in \Sigma$ the global solution $u$ to \eqref{eqn:NLS_Cauchy} with final state $\phi$ satisfies
    \begin{align*}
        u(t) 
            &= e^{it\Delta}\phi + \int_t^\infty e^{i(t-s)\Delta}F(u)(s)~ds\\
            &= e^{it\Delta}\phi + \int_t^\infty e^{i(t-s)\Delta}F(e^{is\Delta}\phi)(s)~ds + r_+(\phi)(t),
    \end{align*}
    where
    \[
        r_+(\phi)(t) = u(t) - e^{it\Delta}\phi - \int_t^\infty e^{i(t-s)\Delta}F(e^{is\Delta}\phi)(s)~ds.
    \]
    Sending $t\to 0$, we obtain
    \[
        \mcal{W}(\phi) = u(0) = \phi + \int_0^\infty e^{-is\Delta}F(e^{is\Delta}\phi)(s)~ds + e_+(\phi)
    \]
    where $e_+(\phi) = r_+(\phi)(0)$.
    A similar expression holds for $\mcal{S}(\phi)$: we have
    \begin{equation}\label{eqn:time_dependent_expansion}
        \mcal{S}(\phi) = [\mcal{S}(\phi) - e^{-it\Delta}u(t)] + \left[ \phi - i\int_0^t e^{-is\Delta}F(e^{is\Delta}\phi)~ds\right] + e^{-it\Delta}r_-(\phi)(t),
    \end{equation}
    where
    \[
        r_-(\phi)(t) = u(t) - e^{it\Delta}\phi + i\int_0^t e^{i(t-s)\Delta}F(e^{is\Delta}\phi)~ds.    
    \]
    By the definition of $\mcal{S}(\phi)$ and Theorem \ref{thm:Sigma_scattering}, we have $\|\mcal{S}(\phi) - e^{-it\Delta}u(\phi)(t)\|_2 \to 0$ as $t\to\infty$.
    Sending $t\to\infty$ in \eqref{eqn:time_dependent_expansion}, we obtain
    \[
        \mcal{S}(\phi) = \phi - i\int_0^\infty e^{-is\Delta}F(e^{is\Delta}\phi)~ds  + e_-(\phi),
    \]
    where $e_-(\phi) = \lim_{t\to\infty} e^{-it\Delta}r_-(\phi)(t)$.
    Therefore we have
    \[
        \|e_\pm(\phi)\|_2 \leq \|r_\pm(\phi)\|_{\infty,2}.    
    \]
    Since $u$ satisfies the integral equation, we may write
    \[
        r_\pm(\phi)(t) = \pm i\int_0^t e^{i(t-s)\Delta} [F(u(\phi)(s)) - F(e^{is\Delta}\phi)]~ds.    
    \]
    By repeated applications of the Strichartz inequality (Proposition \ref{prop:Strichartz}) and H\"older, we obtain
    \begin{align*}
        \|r_\pm(\phi)\|_{\infty,2}
            &= \left\|\int_0^t e^{i(t-s)\Delta} [F(u(\phi)(s)) - F(e^{is\Delta}\phi)]~ds \right\|_{\infty,2}\\
            &\lesssim_{d,p} (\|u\|_{\frac{pq}{q-2},p+2} + \|e^{it\Delta}\phi\|_{\frac{pq}{q-2},p+2})\|u - e^{it\Delta}\phi\|_{q,p+2}\\
            &\lesssim_{d,p} (\|u\|_{\frac{pq}{q-2},p+2}^p + \|e^{it\Delta}\phi\|_{\frac{pq}{q-2},p+2}^p)\|u\|_{\frac{pq}{q-2},p+2}^p\|u\|_{q,p+2}.
    \end{align*}
    Using Lemmas \ref{lem:global_nonlinear_estimate} and \ref{lem:global_linear_estimate} to control the terms in the last line, and noting that $\|\phi\|_\Sigma^{\frac{2}{p+2}} \gg \|\phi\|_\Sigma$ for $\|\phi\|_\Sigma$ small, we obtain Proposition \ref{prop:small_data_expansion}.
\end{proof}

\section{Proof of Theorem \ref{thm:main} for $d\geq 4$}\label{sec:illposedness}

We begin the proof of Theorem \ref{thm:main}.
In the interest of exposition, from here on we restrict to the case $d\geq 4$; this simplifies some technical details, while preserving the essence of the proof.
We refer the reader to the appendix for the modifications needed to recover $d=1,2,3$.

Our first goal is part (1) of Theorem \ref{thm:main}.
First we show that $\mcal{T}_\pm:\Sigma\to L^2$ is of class $C^s(0)$ for all $0<s \leq 1+p$.
Applying Strichartz and arguing as in the proof of Lemma \ref{lem:global_linear_estimate}, we have the estimate
\[
    \left\| \int_0^\infty e^{-is\Delta}F(e^{is\Delta}\phi)~ds\right\|_2 \lesssim \|\phi\|_\Sigma^{1+p}.    
\]
Therefore Proposition \ref{prop:small_data_expansion} gives us
\[
    \mcal{T}_\pm(\phi) - \phi = \mcal{O}_{L^2}(\|\phi\|_\Sigma^{1+p}) + \mcal{O}_{L^2}(\|\phi\|_\Sigma^{\frac{2(2p+1)}{p+2}}) 
\]
whenever $\|\phi\|_\Sigma$ is small.
Noting that $\frac{2(2p+1)}{p+2} > 1+p$ (the condition is equivalent to $p<1$, which holds for mass-subcritical NLS whenever $d\geq 4$), we find that
\[
    \mcal{T}_\pm(\phi) - \phi = \mcal{O}_{L^2}(\|\phi\|_\Sigma^{1+p}).
\]
From this we conclude that $\mcal{T}_\pm:\Sigma\to L^2$ belongs to the class $C^s(0)$ for all $0<s\leq 1+p$.
Moreover, this identifies the first variation of $\mcal{T}_\pm$ at $0$ as $d\mcal{T}_\pm(0)(\phi) = \phi$.

Next we show that $\mcal{T}_\pm:\Sigma\to L^2$ fails to be of class $C^s(0)$ whenever $s>1+p$.
It suffices to show that
\begin{equation}\label{eqn:Sigma_breakdown}
    \mcal{T}_\pm(\phi) - \phi \neq \mcal{O}_{L^2}(\|\phi\|_\Sigma^s)
\end{equation}
as $\|\phi\|_\Sigma\to 0$ for any $s>1+p$; see Lemma \ref{lem:notin_Cs}.

By Proposition \ref{prop:small_data_expansion}, $L^2$ duality, and the unitarity of the linear propagator, we have 
\begin{align*}
    \|\mcal{T}_\pm(\phi) - \phi\|_2
        &\geq \left\|\int_0^\infty e^{-is\Delta}F(e^{is\Delta}\phi)~ds\right\|_2 - \|e_\pm(\phi)\|_2\\
        &\geq \frac{1}{\|\phi\|_2}\left|\int_0^\infty \langle e^{-is\Delta}F(e^{is\Delta}\phi), \phi\rangle_{L_x^2}~ds\right| - \|e_\pm(\phi)\|_2\\
        &= \frac{\|e^{it\Delta}\phi\|_{p+2,p+2}^{p+2}}{\|\phi\|_2} - \|e_\pm(\phi)\|_2.
\end{align*}
Therefore \eqref{eqn:Sigma_breakdown} is proved if we exhibit a sequence $(\phi_n)\subset\Sigma$ with $\|\phi_n\|_\Sigma\to 0$ and
\[
    \frac{\|e^{it\Delta}\phi_n\|_{p+2,p+2}^{p+2}}{\|\phi_n\|_2\|\phi_n\|_\Sigma^s} - \frac{\|e_\pm(\phi_n)\|_\Sigma}{\|\phi_n\|_\Sigma^s} \to \infty.  
\]
Let $\phi\in\Sigma$ with $\|\phi\|_2 = 1$, and for $\ve,\sigma>0$ define
\[
    \phi_{\ve,\sigma}(x) = \frac{\ve}{\sigma^{\frac{d}{2}}}\phi\left(\frac{x}{\sigma}\right).
\]
Then $\phi_{\ve,\sigma}$ satisfies the following scalings:
\[
    \|\phi_{\ve,\sigma}\|_2 = \ve, ~\|\nabla\phi_{\ve,\sigma}\|_2 \sim \frac{\ve}{\sigma}, ~\|\phi_{\ve,\sigma}\|_\Sigma \sim \ve(1+\frac{1}{\sigma} + \sigma),
\]
and
\[
    \|e^{it\Delta}\phi_{\ve,\sigma}\|_{p+2,p+2}^{p+2}
        \sim \ve^{p+2}\sigma^{2-\frac{dp}{2}},   
\]
where for the last expression we have used the parabolic scaling symmetry of the linear Schr\"odinger equation.

We will work in the regime $\ve\ll 1$, $\sigma\gg 1$, and $\ve\sigma\ll 1$.
These together imply that $\|\phi_{\ve,\sigma}\|_\Sigma \sim \ve\sigma\ll 1$, and therefore we are in the small-data regime of Proposition \ref{prop:small_data_expansion}.

By the error estimate \eqref{eqn:error_estimate} of Proposition \ref{prop:small_data_expansion}, we have
\[
    \|e_\pm(\phi_{\ve,\sigma})\|_2 \lesssim  (\ve\sigma)^{\frac{2(2p+1)}{p+2}}.
\]

Next we assume that $\ve = \sigma^{-j}$ with $j>1$.
Since $\sigma\gg 1$, under this assumption that we still have $\ve\sigma\ll 1$.
We now compute:
\begin{align*}
    \frac{\|e^{it\Delta}\phi_{\ve,\sigma}\|_{p+2,p+2}^{p+2}}{\|\phi_{\ve,\sigma}\|_2} - \|e_\pm(\phi_{\ve,\sigma})\|_2
        &\gtrsim \ve^{p+1}\sigma^{2-\frac{dp}{2}} - (\ve\sigma)^{\frac{2(2p+1)}{p+2}}\\
        &= \sigma^{-j(p+1) + 2 - \frac{dp}{2}} - \sigma^{\frac{2(2p+1)}{p+2}(1-j)}.
\end{align*}
We wish for the main term to dominate the error term in the regime $\sigma\gg 1$.
Since we are free to take $j$ arbitrarily large, the main term will dominate provided $p+1 < \frac{2(2p+1)}{p+2}$; as we have already observed, this is automatically satisfied whenever $d\geq 4$.

Therefore we have
\begin{equation}\label{eqn:breakdown_midpoint}
    \|\mcal{T}_\pm(\phi_{\ve,\sigma}) - \phi_{\ve,\sigma}\|_2
        \gtrsim \sigma^{-j(p+1) + 2 - \frac{dp}{2}}
\end{equation}
and thus
\begin{align*}
    \frac{\|\mcal{T}_\pm(\phi_{\ve,\sigma}) - \phi_{\ve,\sigma}\|_2}{\|\phi_{\ve,\sigma}\|_\Sigma^s}
        \gtrsim \sigma^{j[s-(p+1)] + 2 - \frac{dp}{2}-s}.
\end{align*}
Since $s > 1+p$, for $j$ sufficiently large we have $j[s-(p+1)] + 2 - \frac{dp}{2}-s > 0$.
Taking $j$ large to guarantee this inequality and that the main term dominates the error, then taking $\sigma\to\infty$, we find that
\[
    \frac{\|\mcal{T}_\pm(\phi_{\ve,\sigma}) - \phi_{\ve,\sigma}\|_2}{\|\phi_{\ve,\sigma}\|_\Sigma^s} \to \infty.
\]
We thus conclude, as desired, that $\mcal{T}_\pm$ is not of class $C^s(0)$ as a map $\Sigma\to L^2$ whenever $s > 1+p$.

We proceed to part (2) of Theorem \ref{thm:main}.
It suffices to show there exists $0 < \beta < p$ so that
\begin{equation}\label{eqn:L2_breakdown}
    \mcal{T}_\pm(\phi) - \phi \neq \mcal{O}_{L^2}(\|\phi\|_2^{1+\beta})
\end{equation}
as $\|\phi\|_{L^2}\to 0$.
For suppose $\mcal{T}_\pm\in C^{1+\beta}(0)$.
Then $\mcal{T}_\pm(\ve\phi) - \ve a(\phi) = \mcal{O}_{L^2}(\ve^{1+\beta})$ for $\|\phi\|_2 = 1$ and $\ve>0$ small.
Dividing through by $\ve$, noting $\mcal{T}_\pm(0) = 0$, and letting $\ve\to 0$, we find that $a(\phi) = d\mcal{T}_\pm(0)(\phi)$, the first variation of $\mcal{T}_\pm$ at $0$ in the direction $\phi$; but we already know that $d\mcal{T}_\pm(0)(\phi) = \phi$ when $\mcal{T}_\pm$ is regarded as a map $\Sigma\to L^2$, and by density this would be preserved if $\mcal{T}_\pm$ were to admit an extension to $L^2$.
Therefore $\mcal{T}_\pm(\phi) - \phi$ is the only expression that has any hope of satisfying the $\mcal{O}(\|\phi\|_2^{1+\beta})$ bound; showing that this fails proves that $\mcal{T}_\pm \notin C^{1+\beta}(0)$ as a map $L^2\to L^2$.

From here the proof is similar to the proof we gave for \eqref{eqn:Sigma_breakdown}.
Arguing identically as before, \eqref{eqn:L2_breakdown} is proved if we exhibit a sequence $(\phi_n)\subset\Sigma$ with $\|\phi_n\|_\Sigma\to 0$ and
\[
    \frac{\|e^{it\Delta}\phi_n\|_{p+2,p+2}^{p+2}}{\|\phi_n\|_2^{2+\beta}} - \frac{\|e_\pm(\phi_n)\|_2}{\|\phi_n\|^{1+\beta}} \to \infty.  
\]
We take $\sigma\gg 1$, $\ve = \sigma^{-j}$, and $j$ sufficiently large.
Starting from \eqref{eqn:breakdown_midpoint} and dividing through by $\|\phi_{\ve,\sigma}\|_2^{1+\beta}$, we obtain
\begin{align*}
    \frac{\|\mcal{T}_\pm(\phi_{\ve,\sigma}) - \phi_{\ve,\sigma}\|_2}{\|\phi_{\ve,\sigma}\|_2^{1+\beta}}
        \gtrsim \sigma^{j(\beta - p) + 2 - \frac{dp}{2}}
\end{align*}
For this to be large in the regime $\sigma\gg 1$, we require $j(\beta-p) + 2 - \frac{dp}{2} > 0$, or equivalently $\beta > p - \frac{1}{j}(2-\frac{dp}{2})$.
This shows that if $j$ is sufficiently large and this inequality for $\beta$ holds, then $\mcal{T}_\pm$ fails to extend to a map $L^2\to L^2$ of class $C^{1+\beta}(0)$.
Since the constraint on $\beta$ is an open condition, we can optimize by taking the smallest admissible value of $j$, which depends only on $p$ and $d$.
Therefore we have found $j = j(d,p)$ so that if $\beta > p - \frac{1}{j}(2-\frac{dp}{2})$, then $u_+$ fails to extend to a map $L^2\to L^2$ of class $C^{1+\beta}(0)$, which completes the proof of Theorem \ref{thm:main} when $d\geq 4$.
Corollary \ref{cor:main} now follows from Lemma \ref{lem:Cs_necessary}.

\appendix

\section{Proof of Theorem \ref{thm:main} in $d=1,2,3$}\label{app:dim_1_proof}

Here we outline the proof of Theorem \ref{thm:main} in $d=1,2,3$.
There is no truly serious obstruction to be overcome to obtain the result in low dimensions; the choice to break up the proof is entirely for expository purposes, as the proof for $d\geq 4$ is particularly clean and encompasses all of the main ideas.

The main reason why the previous proof does not extend to lower dimensions is that the error estimate
\[
    \|e_\pm(\phi)\|_2 \lesssim \|\phi\|_\Sigma^{\frac{2(2p+1)}{p+2}}    
\]
is no longer strong enough for the main term to dominate the error when $d\leq 3$ and $\alpha(d) < p < \frac{4}{d}$.
The main task is therefore to sharpen this estimate until the error is once again dominated by the main term.

The inefficiency in the above estimate arises from the use of the pseudoconformal energy estimate (Lemma \ref{lem:effective_pseudoconformal}), which is obviously not scaling-invariant and thus leads to losses every time it is invoked.
At this time we do not have another decay-in-time estimate that can replace the pseudoconformal energy estimate, and so we must still take on some losses.
However, noting that there is some slack in the integrability conditions for the time integrals in the proofs of Lemmas \ref{lem:global_nonlinear_estimate} and \ref{lem:global_nonlinear_estimate}, we can at least reduce the total degree to which we do invoke the pseudoconformal energy estimate.

As before, we write $q = \frac{4(p+2)}{dp}$, so that $(q,p+2)$ is an admissible pair.
We now state the sharpened version of \eqref{eqn:error_estimate}:
\begin{prop}\label{prop:error_bound_sharpened}
    Let $d\geq 1$, $\alpha(d) < p < \frac{4}{d}$.
    Define $e_\pm(\phi)$ as before.
    Let $\frac{q-2}{2p} < \eta \leq 1$ and $\frac{1}{2} < \nu \leq 1$.
    Then there exists $\ve = \ve(d,p)>0$ small so that if $\|\phi\|_\Sigma < \ve$, then
    \begin{equation}\label{eqn:Sigma_error_sharpened}
        \|e_\pm(\phi)\|_2 \lesssim_{d,p,\eta,\nu} \|\phi\|_\Sigma^{ Q(d,p,\eta,\nu)}
    \end{equation}
    where
    \[
        Q(d,p,\eta,\nu) = 2p(1-\eta) + (1-\nu) + \frac{2}{p+2}(2\eta p + \nu).
    \]
\end{prop}
We begin the proof.
Write $\theta = 1 - \frac{dp}{2(p+2)}$.
First, we have the following sharpened forms of Lemmas \ref{lem:global_nonlinear_estimate} and \ref{lem:global_linear_estimate}:
\begin{lem}\label{lem:global_nonlinear_estimate_sharpened}
    Let $d\geq 1$, $\alpha(d) < p < \frac{4}{d}$, and $\phi\in\Sigma$.
    Then there exists $\ve = \ve(d,p)>0$ small so that if $\|\phi\|_\Sigma < \ve$, then for $\frac{q-2}{2p} < \eta \leq 1$, we have
    \[
        \|u(\phi)\|_{\frac{pq}{q-2},p+2} \lesssim_{d,p,\eta} (\|\phi\|_2^\theta (E(\phi)^{\frac{1}{2}})^{1-\theta})^{1-\eta}\|\phi\|_\Sigma^{\frac{2\eta}{p+2}},
    \]
    and for $\frac{1}{2} < \nu \leq 1$, we have
    \[
        \|u(\phi)\|_{q,p+2} \lesssim_{d,p,\eta} (\|\phi\|_2^\theta (E(\phi)^{\frac{1}{2}})^{1-\theta})^{1-\nu}\|\phi\|_\Sigma^{\frac{2\nu}{p+2}},
    \]
\end{lem}
\begin{proof}
    We seek to control
    \[
        \left( \int_0^\infty \|u(\phi)(t)\|_{L_x^{p+2}}^{\frac{pq}{q-2}}~dt \right)^{\frac{q-2}{pq}}.    
    \]
    Let $\eta\in[0,1]$.
    We factor the integrand into powers $\|u(\phi)(t)\|_{L_x^{p+2}}^{\frac{pq}{q-2}(1-\eta)}\|u(\phi)(t)\|_{L_x^{p+2}}^{\frac{pq}{q-2}\eta}$.
    We estimate the first piece using Gagliardo-Nirenberg, and the second using the pseudoconformal energy estimate.
    We obtain
    \begin{align*}
        \left( \int_0^\infty \|u(\phi)(t)\|_{L_x^{p+2}}^{\frac{pq}{q-2}}~dt \right)^{\frac{q-2}{pq}}
            &= \left( \int_0^\infty \|u(\phi)(t)\|_{L_x^{p+2}}^{\frac{pq}{q-2}(1-\eta)}\|u(\phi)(t)\|_{L_x^{p+2}}^{\frac{pq}{q-2}\frac{2\eta}{p+2}}~dt \right)^{\frac{q-2}{pq}}\\
            &\leq (\|\phi\|_2^\theta(E(\phi)^\frac{1}{2})^{1-\theta})^{(1-\eta)}\|\phi\|_\Sigma^{\frac{2\eta}{p+2}}\left( \int_0^\infty \langle t\rangle^{-\frac{2p\eta}{q-2}}~dt \right)^{\frac{q-2}{pq}}.
    \end{align*}
    The last integral is finite assuming $\eta > \frac{q-2}{2p}$.
    This establishes the first estimate in Lemma \ref{lem:global_nonlinear_estimate_sharpened}.
    The second estimate for $\|u(\phi)\|_{q,p+2}$ is proved in exactly the same way: we split $\|u(\phi)(t)\|_{L_x^{p+2}} = \|u(\phi)(t)\|_{L_x^{p+2}}^{(1-\nu)}\|u(\phi)(t)\|_{L_x^{p+2}}^\nu$, estimate the first piece using Gagliardo-Nirenberg, and the second by the pseudoconformal energy estimate.
    The condition $\nu>\frac{1}{2}$ is required to make the final integral in time finite.
    We leave the details to the reader.
\end{proof}
\begin{lem}\label{lem:global_linear_estimate_sharpened}
    Let $d\geq 1$, $\alpha(d) < p < \frac{4}{d}$, and $\phi\in\Sigma$.
    Then for $\frac{q-2}{2p} < \eta \leq 1$,
    \[
        \|e^{it\Delta}\phi\|_{\frac{pq}{q-2},p+2} \lesssim_{d,p,\eta} (\|\phi\|_2^\theta \|\nabla\phi\|_2^{1-\theta})^{1-\eta}\|\phi\|_\Sigma^\eta.
    \]
\end{lem}
\begin{proof}
    The proof proceeds almost identically to that of the first part of Lemma \ref{lem:global_nonlinear_estimate_sharpened}.
    As earlier, we factor $\|e^{it\Delta}\phi\|_{p+2} = \|e^{it\Delta}\phi\|_{p+2}^{1-\eta}\|e^{it\Delta}\phi\|_{p+2}^\eta$.
    The first factor can be controlled using Gagliardo-Nirenberg and the conservation of $\dot{H}^s$ norms under the linear Schr\"odinger flow.
    The second factor is controlled near time $0$ by Gagliardo-Nirenberg, and at large times by the dispersive estimate and the embedding $\Sigma\hookrightarrow L^q$ for all $\frac{2d}{d+2} < q \leq 2$.
    The condition $\eta > \frac{q-2}{2p}$ ensures that the time integral that remains is finite.
    We leave the details to the reader.
\end{proof}
\begin{proof}[Proof of Proposition \ref{prop:error_bound_sharpened}]
    We argue as in the proof the error bound in of Proposition \ref{prop:small_data_expansion}, but using Lemmas \ref{lem:global_nonlinear_estimate_sharpened} and \ref{lem:global_linear_estimate_sharpened}.
    Doing so, we arrive at an estimate of the form
    \[
        \|e_\pm(\phi)\|_2 \lesssim_{d,p,\eta,\nu} \|\phi\|_2^\alpha (E(\phi)^\frac{1}{2})^{\beta+\delta} \|\phi\|_\Sigma^\gamma + \|\phi\|_2^\alpha(E(\phi)^\frac{1}{2})^\beta\|\nabla\phi\|_2^\delta \|\phi\|_\Sigma^{\gamma + \frac{p}{p+2}\eta p},
    \]
    where:
    \begin{align*}
        \alpha &= \theta(2p(1-\eta) + (1-\nu));\\
        \beta &= (1-\theta)(p(1-\eta) + (1-\nu));\\
        \gamma &= \frac{2}{p+2}(2\eta p + \nu);\\
        \delta &= (1-\theta)p(1-\eta).
    \end{align*}
    Note that $Q(d,p,\eta,\nu) = \alpha+\beta+\delta+\gamma$.
    By Sobolev embedding, $E(\phi)$ is controlled by $\|\phi\|_\Sigma^2 + \|\phi\|_\Sigma^{p+2}$, and by the assumption $\|\phi\|_\Sigma\ll 1$ the second term is negligible.
    Therefore every norm and each $(E(\phi))^{\frac{1}{2}}$ is majorized by $\|\phi\|_\Sigma$, and we have:
    \[
        \|e_\pm(\phi)\|_2 
            \lesssim_{d,p,\eta,\nu} \|\phi\|_\Sigma^{Q(d,p,\eta,\nu)} + \|\phi\|_\Sigma^{Q(d,p,\eta,\nu) + \frac{p}{p+2}\eta p}
            \lesssim \|\phi\|_\Sigma^{Q(d,p,\eta,\nu)}. \qedhere
    \]
\end{proof}
We are now ready to prove Theorem \ref{thm:main} in full generality.
\begin{proof}[Proof of Theorem \ref{thm:main}]
    We mention only the necessary changes relative to the proof in dimensions $d\geq 4$.

    The first step is to prove that $\mcal{T}_\pm:\Sigma\to L^2$ is of class $C^s(0)$ for all $0<s<1+p$.
    It suffices as before to show that
    \[
        \mcal{T}_\pm(\phi) - \phi = \mcal{O}(\|\phi\|_\Sigma^{1+p})
    \]
    whenever $\|\phi\|_\Sigma$ is small.
    To ensure this we must show that $e_+(\phi)$ is of higher order in $\|\phi\|_\Sigma$ than the main term, i.e. $1 + p < Q(d,p,\eta,\nu)$ for some admissible choice of $\eta$ and $\nu$.
    Since $\eta$ can be arbitrarily close to $\frac{q-2}{2p}$ and $\nu$ can be arbitrarily close to $\frac{1}{2}$, it suffices to show that
    \[
        1 + p < Q\left(d,p,\frac{q-2}{2p},\frac{1}{2}\right).
    \]
    This is equivalent to the condition
    \[
        2dp^2 + (11d-8)p + (8d-16) > 0.    
    \]
    When $d\geq 2$, this is automatically satisfied for $p>0$ because the coefficients are nonnegative.
    When $d=1$, the positive root of this polynomial is smaller than $\frac{3}{2}$, and thus this is satisfied for $p > 2 = \frac{2}{d}$.

    Next we show that $\mcal{T}_\pm:\Sigma\to L^2$ is not of class $C^s(0)$ for $s > 1+p$, and does not extend to a map $L^2\to L^2$ of class $C^{1+\beta}(0)$ for some $0<\beta<p$.
    As before it suffices to show that
    \begin{equation}\label{eqn:Sigma_illposed_low_dim}
        \mcal{T}_\pm(\phi) - \phi \neq \mcal{O}_{L^2}(\|\phi\|_\Sigma^s)  
    \end{equation}
    in the first case, and
    \begin{equation}\label{eqn:L2_illposed_low_dim}
        \mcal{T}_\pm(\phi) - \phi \neq \mcal{O}_{L^2}(\|\phi\|_{L^2}^{1+\beta}).  
    \end{equation}
    in the latter case.
    Examining the proof in $d\geq 4$, we observe that the only way in which the size of $e_\pm(\phi)$ enters into either argument is to show that there exists a regime $\ve\ll 1$, $\sigma\gg 1$ and $\ve\sigma\ll 1$ so that $\|\phi_{\ve,\sigma}\|_2$ (where $\phi_{\ve,\sigma}$ is defined as before) is dominated by the main term $\|\phi_{\ve,\sigma}\|_2^{-1}\|e^{it\Delta}\phi_{\ve,\sigma}\|_{p+2,p+2}^{p+2}$.
    Taking $\ve = \sigma^{-j}$ with $j>1$ to be determined, the main term is still of size
    \[
        \frac{\|e^{it\Delta}\phi_{\ve,\sigma}\|_{p+2,p+2}^{p+2}}{\|\phi_{\ve,\sigma}\|_2} \sim \sigma^{-j(p+1)+2-\frac{dp}{2}}.    
    \]
    We use \eqref{eqn:Sigma_error_sharpened} to control the error by
    \[
        \|e_\pm(\phi_{\ve,\sigma})\|_2 \lesssim_{d,p,\eta,\nu} \sigma^{(1-j)Q(d,p,\eta,\nu)}.
    \]
    Noting as before that $Q(d,p,\eta,\nu) > p+1$ for a judicious choice of $\eta$ and $\nu$, we see that $\|e_\pm(\phi_{\ve,\sigma})\|_\Sigma$ is negligible relative to the main term for $j$ sufficiently large and $\sigma\gg 1$.
    From here the proof of \eqref{eqn:Sigma_illposed_low_dim} and \eqref{eqn:L2_illposed_low_dim} proceeds exactly as when $d\geq 4$.
\end{proof}

\section{Pointwise H\"older spaces and Gateaux derivatives}\label{app:Holder_spaces}
In this appendix we relate the notion of pointwise H\"older regularity given in Definition \ref{dfn:Holder_space} to more familiar notions.
For convenience we reproduce the definition here:
\begin{dfn}[Pointwise H\"older space \cites{Andersson97}]
    Let $X$ and $Y$ be Banach spaces.
    Let $x_0\in X$ and $U$ a convex open neighborhood of $x_0$.
    Fix $s>0$, and let $n$ be the integer part of $s$.
    For $s>0$, we say that the map $G:X\to Y$ belongs to the \ita{pointwise H\"older space} $C^s(x_0)$ if for all $h\in X$ with $\|h\|_X = 1$, there exist coefficients $\{a_j(x_0; h)\}_{j=0}^n \subset Y$ such that
    \[
        \|G(x_0 + \ve h) - G(x_0) - \sum_{j=1}^n \ve^j a_j(x_0; h)\|_Y \lesssim \ve^s   
    \]
    for all $\ve>0$ sufficiently small, with the implicit constant independent of the direction $h$.
\end{dfn}
This is related to two notions: the Peano derivative (also known as the de la Vall\'ee-Poussin derivative), and the Gateaux derivative.
\begin{dfn}[Peano, de la Vall\'ee-Poussin derivative]
    Let $X$ and $Y$ be Banach spaces.
    Let $x_0\in X$, let $U$ be a convex open neighborhood of $x_0$, and let $h\in X$ with $\|h\|_X = 1$.
    For $n\geq 1$, we say that a map $G:U\to Y$ has an $n$-th \ita{Peano derivative}, or \ita{de la Vall\'ee-Poussin derivative}, at $x_0$ in the direction $h$ if there exist $\{a_j(x_0;h)\}_{j=1}^n \subset Y$ such that
    \[
        \|G(x_0 + \ve h) - G(x_0) - \sum_{j=1}^n \frac{1}{j!}\ve^ja_j(x_0;h)\|_Y = o(\ve^n;h)
    \]
    as $\ve\to 0$.
\end{dfn}
Therefore if $G\in C^s(x_0)$ with $s\geq n$, then $G$ automatically has an $n$-th Peano derivative, with an asymptotic bound as $\ve\to 0$ which is uniform in $h$; moreover, if $s>n$, then the asymptotic bound is stronger.
\begin{dfn}[Gateaux derivative \cites{Graves27}]
    Let $X$ and $Y$ be Banach spaces.
    Let $x_0\in X$ and $U\subset X$ a convex neighborhood of $x_0$.
    We say that the map $G:U\to Y$ is \ita{Gateaux differentiable} at $x_0$ in the direction $h\in X$ if the limit
    \[
        dG(x_0;h) = \lim_{\ve \to 0^+} \frac{G(x_0+\ve h) - G(x_0)}{\ve} = \frac{d}{d\ve}\bigg|_{\ve=0} G(x_0+\ve h)   
    \]
    exists in $Y$.
    In that case, we call $dG(x_0;h)$ the \ita{Gateaux derivative}, or \ita{first variation}, of $G$ at $u$ in the direction $v$.
    If $dG(x_0;h)$ exists for all $h\in X$, we say that $G$ is Gateaux differentiable at $x_0$.
    Similarly, we define the Gateaux derivative of order $n$, or $n$-th variation, by
    \[
        d^nG(x_0;h) = \frac{d^n}{d\ve^n}\bigg|_{\ve=0} G(x_0+\ve h).
    \]
\end{dfn}
Gateaux derivatives are homogeneous in their second argument: $d^jG(x_0;\ve h) = \ve^jd^jG(x_0;h)$ for all $\ve\in\bb{R}$ (\cites{Graves27}, Lemma 1.2).

It is clear that if $n\geq 1$ and $G:U\to Y$ has an $n$-th Peano derivative $a_n(x_0;h)$ at $x_0$ in the direction $h$, then it also has $j$-th Peano derivatives $a_j(x_0;h)$ at $x_0$ in the direction $h$ for $j=1,\ldots,n-1$; moreover, $G$ is Gateaux differentiable at $x_0$ in the direction $h$ with first variation $dG(x_0;h) = a_1(x_0;h)$.
It is not, however, true that $G$ has variations of any higher order, even in the real-valued case: a counterexample is $f(x) = x^3\sin(1/x)$ for $x\neq 0$, $f(0) = 0$, for which the second Peano derivative exists at $0$, but not $f''(0)$ \cites{Oliver54}.
For this reason, $C^s(x_0)$ is not exactly a replacement for the space of $n$-times Gateaux differentiable maps with $d^nG(x_0;h)$ H\"older continuous of order $s-n$ in $x_0$.
When $s > 2$, we are not even able to detect from the definition whether a map in $C^s(x_0)$ has a second variation at $x_0$.
However, $C^s(x_0)$ is still a useful notion for detecting when a map \ita{fails} to have a certain level of Gateaux regularity, which is what is relevant for the breakdown of regularity statements in Corollary \ref{cor:main}.
This arises through the generalization of Taylor's theorem with remainder for Banach space valued functions.
\begin{thm}[Taylor's theorem with remainder; \cites{Graves27}, Theorem 5]\label{thm:Taylor_with_remainder}
    Let $X$ and $Y$ be Banach spaces.
    Let $U\subset X$ be a convex neighborhood of $u\in X$.
    Let $G:U\to Y$ be $n$-times Gateaux differentiable on $U$, and let $x_0\in X$ be such that $d^nG(x_0 + s\ve h;h)$ is Riemann integrable (defined in \cites{Graves27}) over $s\in(0,1)$ whenever $\ve>0$ is sufficiently small.
    Then for all $h\in X$ with $\|h\|_X=1$ and $\ve>0$ small,
    \[
        G(x_0+ \ve h) = G(x_0) + \sum_{j=1}^n \frac{\ve^j}{j!}d^jG(x_0;h) + \ve^{n+1}R_{n+1}(x_0,h,\ve)
    \]
    where
    \[
        R_{n+1}(x_0,h,\ve) = \frac{1}{n!}\int_0^1 (1-s)^n d^{n+1}G(x_0+s\ve h;h)~ds.    
    \]
\end{thm}
We now arrive at the main statement of interest.
It states that for $n < s < n+1$, membership in $C^s(x_0)$ is necessary for a map $G$ to be $n$ times Gateaux differentiable with $d^nG(x;h)$ H\"older continuous of order $s-n$.
This gives us a way of detecting whether $G$ admits $s$ derivatives in this latter sense.
\begin{lem}
    Let $X$ and $Y$ be Banach spaces.
    Let $U\subset X$ be a convex neighborhood of $x_0\in X$.
    Let $G:U\to Y$ be a map, and suppose $G\notin C^s(x_0)$ with $n < s < n+1$.
    Then $d^nG(x;h)$, if it exists for $x\in U$, cannot be a H\"older continuous function of $x$ of order $s-n$ with H\"older seminorm uniformly bounded in $h$.
\end{lem}
\begin{proof}
    Suppose for contradiction that $d^nG(x;h)$ exists on $U$ and is H\"older continuous of order $s-n$ in $x$, with H\"older seminorm uniformly bounded in $h$.
    Then all lower order Gateaux derivatives must also exist.
    This implies that $G$ satisfies the conditions of Theorem \ref{thm:Taylor_with_remainder}, and hence admits the expansion
    \[
        G(x_0+\ve h) = G(x_0) + \sum_{j=1}^{n-1} \frac{\ve^j}{j!}d^jG(x_0;h) + \ve^nR_n(x_0,h,\ve)    
    \]
    as $\ve\to 0$, where $R_n$ is given as in Theorem \ref{thm:Taylor_with_remainder}.
    By the H\"older continuity assumption, we have
    \begin{align*}
        \|R_n(x_0,h,\ve) &- \frac{1}{n!}d^nG(x_0;h)\|_Y\\
            &= \left\|\frac{1}{(n-1)!}\int_0^1 (1-r)^{n-1}[d^nG(x_0 + r\ve h;h) - d^nG(x_0;h)]~dr\right\|_Y\\
            &\leq \frac{1}{(n-1)!}\int_0^1 (1-r)^{n-1}\|d^nG(x_0+r\ve h;h) - d^nG(x_0;h)\|_Y~dr\\
            &\lesssim \ve^{s-n}\int_0^1 (1-r)^{n-1}r^s~dr \leq \ve^{s-n}.
    \end{align*}
    Therefore
    \begin{align*}
        G(x_0+\ve h) 
            &= G(x_0) + \sum_{j=1}^{n-1} \frac{\ve^j}{j!}d^jG(x_0;h) + \ve^nR_n(x_0,h,\ve)\\
            &= G(x_0) + \sum_{j=1}^n \frac{\ve^j}{j!}d^jG(x_0;h) + \ve^n[R_n(x_0,h,\ve)-\frac{1}{n!}d^nG(x_0;h)]\\
            &= G(x_0) + \sum_{j=1}^n \frac{\ve^j}{j!}d^jG(x_0;h) + \mcal{O}_Y(\ve^s).
    \end{align*}
    But then $G\in C^s(x_0)$, contradiction.
\end{proof}

Lastly, we need a way of checking that a given $G$ does not belong to the class $C^s(x_0)$.
\begin{lem}\label{lem:notin_Cs}
    Let $n$ be a positive integer, and let $n < s < s+\delta < n+1$.
    Assume $G\in C^s(x_0)$ with Peano derivatives $\{a_j(x_0;h)\}_{j=1}^n$, so that
    \[
        \|G(x_0 + \ve h) - G(x_0) - \sum_{j=1}^n \ve^j a_j(x_0; h)\|_Y \lesssim \ve^s.
    \]
    Suppose also that
    \[
        \|G(x_0 + \ve h) - G(x_0) - \sum_{j=1}^n \ve^j a_j(x_0; h)\|_Y \not\lesssim \ve^{s+\delta}.
    \]
    Then $G\notin C^{s+\delta}(x_0)$.
\end{lem}
The proof is based on the following uniqueness statement for the Peano derivatives:
\begin{thm}[\cites{Graves27}, Theorem 6]\label{thm:Taylor_converse}
    Let $X$ and $Y$ be Banach spaces.
    Let $U\subset X$ be a convex neighborhood of $x_0\in X$.
    Let $G:U\to Y$ be a map.
    Then for each positive integer $n$, there exists at most one expansion of the form
    \[
        G(x_0+h) = G(x_0) + \sum_{j=1}^n a_j(x_0;h) + R_{n+1}(x_0,h)
    \]
    satisfying $a_j(x_0;sh) = s^ja_j(x_0;h)$ and $R_{n+1}(x_0,h) = o(\|h\|_Y^n)$ as $h\to 0$.
\end{thm}
\begin{proof}[Proof of Lemma \ref{lem:notin_Cs}]
    Suppose to the contrary that $G\in C^{s+\delta}(x_0)$.
    Then there are coefficients $\{b_j(x_0;h)\}_{j=1}^n$ such that
    \[
        \|G(x_0 + \ve h) - G(x_0) - \sum_{j=1}^n \ve^j b_j(x_0; h)\|_Y \lesssim \ve^{s+\delta}.
    \]
    Then we have two polynomial expansions for $G(x_0+h)$ around $x_0$ of degree $n$ with $o(\|h\|_Y^n)$ remainder as $h\to 0$.
    By Theorem \ref{thm:Taylor_converse}, it follows that $b_j = a_j$.
    But this contradicts the assumption that the the error in the expansion $G(x_0+\ve h) \sim G(x_0) + \sum_{j=1}^n \ve^ja_j(x_0;h)$ is not $\mcal{O}(\ve^{s+\delta})$.
\end{proof}
The utility of Lemma \ref{lem:notin_Cs} is that so long as we can verify one asymptotically valid polynomial approximation of $G(x_0+\ve h)$, the same polynomial approximation can be used to check the membership of $G$ in $C^s(x_0)$, as long as there is no need to add a higher-order derivative term to the expansion.

\section*{Acknowledgments} The author thanks his advisors Rowan Killip and Monica Vi\c san for many helpful discussions and guidance.
This work was supported by NSF grants 1600942 (principal investigator: Rowan Killip) and 1500707 (principal investigator: Monica Vi\c{s}an).

\bibliography{main}

\begin{bibdiv}
\begin{biblist}

\bib{Andersson97}{article}{
      author={Andersson, P.},
       title={Characterization of pointwise {H}\"{o}lder regularity},
        date={1997},
        ISSN={1063-5203},
     journal={Appl. Comput. Harmon. Anal.},
      volume={4},
      number={4},
       pages={429\ndash 443},
         url={https://doi.org/10.1006/acha.1997.0219},
      review={\MR{1474098}},
}

\bib{Barab84}{article}{
      author={Barab, J.~E.},
       title={Nonexistence of asymptotically free solutions for a nonlinear
  {S}chr\"odinger equation},
        date={1984},
        ISSN={0022-2488},
     journal={J. Math. Phys.},
      volume={25},
      number={11},
       pages={3270\ndash 3273},
         url={http://dx.doi.org/10.1063/1.526074},
      review={\MR{761850}},
}

\bib{BeTa06}{article}{
      author={Bejenaru, I.},
      author={Tao, T.},
       title={Sharp well-posedness and ill-posedness results for a quadratic
  non-linear {S}chr\"{o}dinger equation},
        date={2006},
        ISSN={0022-1236},
     journal={J. Funct. Anal.},
      volume={233},
      number={1},
       pages={228\ndash 259},
         url={https://doi.org/10.1016/j.jfa.2005.08.004},
      review={\MR{2204680}},
}

\bib{Bo97}{article}{
      author={Bourgain, J.},
       title={Periodic {K}orteweg de {V}ries equation with measures as initial
  data},
        date={1997},
        ISSN={1022-1824},
     journal={Selecta Math. (N.S.)},
      volume={3},
      number={2},
       pages={115\ndash 159},
         url={https://doi.org/10.1007/s000290050008},
      review={\MR{1466164}},
}

\bib{CaOz08}{article}{
      author={Carles, R.},
      author={Ozawa, T.},
       title={On the wave operators for the critical nonlinear
  {S}chr\"{o}dinger equation},
        date={2008},
        ISSN={1073-2780},
     journal={Math. Res. Lett.},
      volume={15},
      number={1},
       pages={185\ndash 195},
         url={https://doi.org/10.4310/MRL.2008.v15.n1.a15},
      review={\MR{2367183}},
}

\bib{CaWe92}{article}{
      author={Cazenave, T.},
      author={Weissler, F.~B.},
       title={Rapidly decaying solutions of the nonlinear {S}chr\"{o}dinger
  equation},
        date={1992},
        ISSN={0010-3616},
     journal={Comm. Math. Phys.},
      volume={147},
      number={1},
       pages={75\ndash 100},
         url={http://projecteuclid.org/euclid.cmp/1104250527},
      review={\MR{1171761}},
}

\bib{ChCoTa03}{article}{
      author={{Christ}, M.},
      author={{Colliander}, J.},
      author={{Tao}, T.},
       title={{Ill-posedness for nonlinear Schr\"odinger and wave equations}},
        date={2003-11},
     journal={ArXiv Mathematics e-prints},
      eprint={math/0311048},
}

\bib{Do12}{article}{
      author={Dodson, B.},
       title={Global well-posedness and scattering for the defocusing,
  {$L^{2}$}-critical nonlinear {S}chr\"{o}dinger equation when {$d\geq3$}},
        date={2012},
        ISSN={0894-0347},
     journal={J. Amer. Math. Soc.},
      volume={25},
      number={2},
       pages={429\ndash 463},
         url={https://doi.org/10.1090/S0894-0347-2011-00727-3},
      review={\MR{2869023}},
}

\bib{Do16.d1GWP}{article}{
      author={Dodson, B.},
       title={Global well-posedness and scattering for the defocusing, {$L^2$}
  critical, nonlinear {S}chr\"{o}dinger equation when {$d=1$}},
        date={2016},
        ISSN={0002-9327},
     journal={Amer. J. Math.},
      volume={138},
      number={2},
       pages={531\ndash 569},
         url={https://doi.org/10.1353/ajm.2016.0016},
      review={\MR{3483476}},
}

\bib{Do16.d2GWP}{article}{
      author={Dodson, B.},
       title={Global well-posedness and scattering for the defocusing,
  {$L^2$}-critical, nonlinear {S}chr\"{o}dinger equation when {$d=2$}},
        date={2016},
        ISSN={0012-7094},
     journal={Duke Math. J.},
      volume={165},
      number={18},
       pages={3435\ndash 3516},
         url={https://doi.org/10.1215/00127094-3673888},
      review={\MR{3577369}},
}

\bib{GiVe79.Cauchy}{article}{
      author={Ginibre, J.},
      author={Velo, G.},
       title={On a class of nonlinear {S}chr\"odinger equations. {I}. {T}he
  {C}auchy problem, general case},
        date={1979},
        ISSN={0022-1236},
     journal={J. Funct. Anal.},
      volume={32},
      number={1},
       pages={1\ndash 32},
         url={http://dx.doi.org/10.1016/0022-1236(79)90076-4},
      review={\MR{533218}},
}

\bib{GiVe19.Scattering}{article}{
      author={Ginibre, J.},
      author={Velo, G.},
       title={On a class of nonlinear {S}chr\"odinger equations. {II}.
  {S}cattering theory, general case},
        date={1979},
        ISSN={0022-1236},
     journal={J. Funct. Anal.},
      volume={32},
      number={1},
       pages={33\ndash 71},
         url={http://dx.doi.org/10.1016/0022-1236(79)90077-6},
      review={\MR{533219}},
}

\bib{Glassey73}{article}{
      author={Glassey, R.~T.},
       title={On the asymptotic behavior of nonlinear wave equations},
        date={1973},
        ISSN={0002-9947},
     journal={Trans. Amer. Math. Soc.},
      volume={182},
       pages={187\ndash 200},
         url={http://dx.doi.org/10.2307/1996529},
      review={\MR{0330782}},
}

\bib{Graves27}{article}{
      author={Graves, L.~M.},
       title={Riemann integration and {T}aylor's theorem in general analysis},
        date={1927},
        ISSN={0002-9947},
     journal={Trans. Amer. Math. Soc.},
      volume={29},
      number={1},
       pages={163\ndash 177},
         url={https://doi.org/10.2307/1989284},
      review={\MR{1501382}},
}

\bib{Ki18}{article}{
      author={{Kishimoto}, N.},
       title={{A remark on norm inflation for nonlinear Schr\"odinger
  equations}},
        date={2018-06},
     journal={ArXiv e-prints},
      eprint={1806.10066},
}

\bib{Murphy.scattering}{misc}{
      author={Murphy, J.},
       title={Subcritical scattering for defocusing nonlinear {S}chr\"odinger
  equations},
        note={http://web.mst.edu/~jcmcfd/expository.pdf},
}

\bib{Na01}{article}{
      author={Nakanishi, K.},
       title={Asymptotically-free solutions for the short-range nonlinear
  schrödinger equation},
        date={2001},
     journal={SIAM Journal on Mathematical Analysis},
      volume={32},
      number={6},
       pages={1265\ndash 1271},
      eprint={https://doi.org/10.1137/S0036141000369083},
         url={https://doi.org/10.1137/S0036141000369083},
}

\bib{Oliver54}{article}{
      author={Oliver, H.~W.},
       title={The exact {P}eano derivative},
        date={1954},
        ISSN={0002-9947},
     journal={Trans. Amer. Math. Soc.},
      volume={76},
       pages={444\ndash 456},
         url={https://doi.org/10.2307/1990791},
      review={\MR{62207}},
}

\bib{Strauss73}{inproceedings}{
      author={Strauss, W.},
       title={Nonlinear scattering theory},
        date={1973},
   booktitle={{Scattering Theory in Mathematical Physics}},
      editor={Lavita, J.~A.},
      editor={Marchand, J.~P.},
       pages={53 \ndash  78},
}

\bib{Tao06.Dispersive}{book}{
      author={Tao, T.},
       title={Nonlinear dispersive equations},
      series={CBMS Regional Conference Series in Mathematics},
   publisher={Published for the Conference Board of the Mathematical Sciences,
  Washington, DC; by the American Mathematical Society, Providence, RI},
        date={2006},
      volume={106},
        ISBN={0-8218-4143-2},
         url={https://doi.org/10.1090/cbms/106},
        note={Local and global analysis},
      review={\MR{2233925}},
}

\bib{Tao09}{article}{
      author={Tao, T.},
       title={A pseudoconformal compactification of the nonlinear
  {S}chr\"{o}dinger equation and applications},
        date={2009},
        ISSN={1076-9803},
     journal={New York J. Math.},
      volume={15},
       pages={265\ndash 282},
         url={http://nyjm.albany.edu:8000/j/2009/15_265.html},
      review={\MR{2530148}},
}

\bib{TsYa84}{article}{
      author={Tsutsumi, Y.},
      author={Yajima, K.},
       title={The asymptotic behavior of nonlinear {S}chr\"odinger equations},
        date={1984},
        ISSN={0273-0979},
     journal={Bull. Amer. Math. Soc. (N.S.)},
      volume={11},
      number={1},
       pages={186\ndash 188},
         url={http://dx.doi.org/10.1090/S0273-0979-1984-15263-7},
      review={\MR{741737}},
}

\end{biblist}
\end{bibdiv}

\end{document}